\pdfoutput=1

\documentclass{amsart}

\usepackage{latexsym}
\usepackage{amssymb,amsthm,amsmath,amsxtra}
\usepackage{color}
\usepackage{fullpage}
\usepackage{comment}
\usepackage{enumerate}
\usepackage{graphics}
\usepackage{marginnote}
\usepackage{adjustbox}
\usepackage[section]{placeins}
\usepackage{cancel}
\usepackage[space]{grffile}

\theoremstyle{plain}
\newtheorem{theorem}{Theorem}[section]

\newtheorem{lemma}[theorem]{Lemma}

\theoremstyle{definition}

\theoremstyle{remark}
\newtheorem{rem}{Remark}[section]

\usepackage[all]{xy}

\theoremstyle{remark}

\usepackage{amssymb}
\usepackage{amsthm}
\usepackage{amsmath}
\usepackage{amsxtra}
\usepackage[foot]{amsaddr}

\newcommand{\Q}{\mathbb Q}
\newcommand{\R}{\mathbb R}

\newcommand{\C}{\mathbb C}

\newcommand{\Z}{\mathbb Z}

\begin{document}
\title{Origami Constructions of Rings of Integers of Imaginary Quadratic Fields}

\author{J\"{u}rgen Kritschgau$^1$}
\address{$^1$Department of Mathematics\\ Iowa State University\\ 396 Carver Hall, 411 Morrill Road\\Ames, IA 50011}
\author{Adriana Salerno$^2$}
\address{$^2$Department of Mathematics\\ Bates College\\ 3 Andrews Rd\\ Lewiston, ME 04240}

\email{jkritsch@iastate.edu, asalerno@bates.edu}

\begin{abstract}
In the making of origami, one starts with a piece of paper, and through a series of folds along seed points one constructs complicated three-dimensional shapes. Mathematically, one can think of the complex numbers as representing the piece of paper, and the seed points and folds as a way to generate a subset of the complex numbers. Under certain constraints, this construction can give rise to a ring, which we call an origami ring. We will talk about the basic construction of an origami ring and further extensions and implications of these ideas in algebra and number theory, extending results of Buhler, et.al. In particular, in this paper we show that it is possible to obtain the ring of integers of an imaginary quadratic field through an origami construction. 
\end{abstract}


\maketitle

\section{Introduction}

In origami, the artist uses intersections of folds as reference points to make new folds. This kind of construction can be extended to points on the complex plane. In \cite{Buhler2010}, the authors define one such mathematical construction. In this construction, one can think of the complex plane as representing the ``paper", and lines representing the ``folds". The question they explored is which points in the plane can be constructed through iterated intersections of lines, starting with a prescribed set of allowable angles and only the points 0 and 1. 

First, we say that the set $S=\{0,1\}$ is the set of \emph{seed points}. We fix a set $U$ of angles, or ``directions", determining which lines we can draw through the points in our set. Thus, we can define the ``fold" through the point $p$ with angle $u$ as the line given by $$L_u(p):=\{p+ru:r\in\R\}.$$

Notice that $U$ can also be comprised of points (thinking of $u\in U$ as defining a direction) on the unit circle, i.e. the circle group $\mathbb{T}$. Moreover, $u$ and $-u$ define the same line, so we can think of the directions as being in the quotient group $\mathbb{T}/\{\pm1\}$. 

Finally, if $u$ and $v$ in $U$ determine distinct folds, we say that $$I_{u,v}(p,q)=L_u(p)\cap L_v(q)$$ is the unique point of intersection of the lines $L_u(p)$ and  $L_v(q)$. 

Let $R(U)$ be the set of points obtained by iterated intersections $I_{u,v}(p,q)$, starting with the points in $S$. Alternatively, we define $R(U)$ to be the smallest subset of $\mathbb{C}$ that contains $0$ and $1$ and $I_{u,v}(p,q)$ whenever it contains $p$ and $q$, and $u, v$ determine distinct folds. The main theorem of \cite{Buhler2010} is the following.

\begin{theorem}\label{Buhler}
If $U$ is a subgroup of $\mathbb{T}/\{\pm1\}$, and $|U|\geq3$, then $R(U)$ is a subring of $\mathbb{C}$. 
\end{theorem}

Let $U_n$ denote the cyclic group of order $n$ generated by $e^{i\pi/n}\mod\{\pm1\}$.  Then Buhler, et. al., obtain the following corollary. 

\begin{theorem}
Let $n\geq3$. If $n$ is prime, then $R(U_n)=\mathbb{Z}[\zeta_n]$ is the cyclotomic integer ring. If $n$ is not prime, then $R(U_n)=\mathbb{Z}[\zeta_n,\frac{1}{n}]$. 
\end{theorem} 

In \cite{nedrenco}, Nedrenco explores whether $R(U)$ is a subring of $\C$ even if $U$ is not a group, and obtains a negative answer, and some necessary conditions for this to be true. The main result is that given the set of  directions $U=\{1,e^{i\alpha},e^{i\beta}\}$, if $\alpha\not\equiv\beta\bmod\pi$ then $R(U)=\Z+z\Z$ for some $z\in\C$. Clearly, this will not always be a ring.

In this paper, we explore the inverse problem, that is, given an ``interesting" subring of $\C$, can we obtain it via an origami construction? The answer is affirmative in the case of the ring of integers of an imaginary quadratic field. 

The next section of this paper delves deeper into the origami construction, in particular the intersection operator. Some properties in this section are crucial for understanding the proofs of our main results. This section also explores an example of an origami construction in more depth, that of the Gaussian integers, since it illustrates the geometric and algebraic approach through a very well known ring. Finally, in Section 3, we state and prove our main result. 

\begin{rem}
Even though Nedrenco essentially proved one of our two results in \cite{nedrenco}, this was accomplished completely independently by us. In fact, our proof is different, since we are interested in the reverse of Nedrenco's question. That is, we explore whether a given ring can be an origami ring. Nedrenco explores the conditions for which his origami construction ends up being a ring of algebraic integers. This distinction is subtle, but important, and we want to clarify that all of what follows, unless otherwise indicated as coming from \cite{Buhler2010}, is original work. 
\end{rem} 

\subsection*{Acknowledgements.} This work was part of the first author's senior Honors Thesis at Bates College, advised by the second author, and we thank the Bates Mathematics Department, in particular Peter Wong, for useful feedback and discussions throughout the thesis process. We would also like to thank the rest of the Honors committee, Pamela Harris and Matthew Cot\'{e}, for the careful reading of the thesis and their insightful questions. Finally, we would like to thank the 2015 Summer Undergraduate Applied Mathematics Institute (SUAMI) at Carnegie Mellon University for inspiring the research problem.

\section{Properties of Origami Rings}

\subsection{The intersection operator} 

Let $U\subset \mathbb T$, as before. There are important properties of the $I_{u,v}(p,q)$ operator that are integral for us to prove our theorem. 

\par Let $u,v\in U$ be two distinct angles. Let $p,q$ be points in $R(U)$. Consider the pair of intersecting lines $L_u(p)$ and $L_v(q)$. In \cite{Buhler2010}, it is shown that we can express $I_{u,v}(p,q)$ as 
\begin{equation}\tag{$\ast$}
I_{u,v}(p,q)=\frac{u\bar pv-\bar upv}{u\bar v-\bar uv}+\frac{q\bar vu-\bar qvu}{\bar uv-u\bar v}=\frac{[u,p]}{[u,v]}v+\frac{[v,q]}{[v,u]}u
\label{closedform}
\end{equation}
where $[x,y]=\overline{x}y-x\overline{y}$. 

From the algebraic closed form (\ref{closedform}) of the intersection operator, we can see by straightforward computation that the following properties hold for for $p,q,u,v\in \C$.

\begin{description}
\item[Symmetry] $I_{u,v}(p,q)=I_{v,u}(q,p)$
\item[Reduction] $I_{u,v}(p,q)=I_{u,v}(p,0)+I_{v,u}(q,0)$
\item[Linearity] $I_{u,v}(p+q,0)=I_{u,v}(p,0)+I_{u,v}(q,0)$ and $rI_{u,v}(p,0)=I_{u,v}(rp,0)$ where $r\in\R$.
\item[Projection] $I_{u,v}(p,0)$ is a projection of $p$ on the line $\{rv:r\in\R\}$ in the $u$ direction.
\item[Rotation] For $w\in \mathbb{T}$, $wI_{u,v}(p,q)=I_{wu,wv}(wp,wq)$.
\end{description}

\subsection{An illustrative example}

Let $S=\{0,1\}$ be our set of seed points. Now, let $U=\{1,e^{i\pi/4},i\}$. This is clearly not a group, since $e^{\frac{i\pi}{4}}i=e^{\frac{3i\pi}{4}}\not\in U$. In figure \ref{fig: gauss}, we show the different stages of the construction, obtained by iterated intersections. 

\begin{figure}[!h]
  \includegraphics[height=.3\textheight]{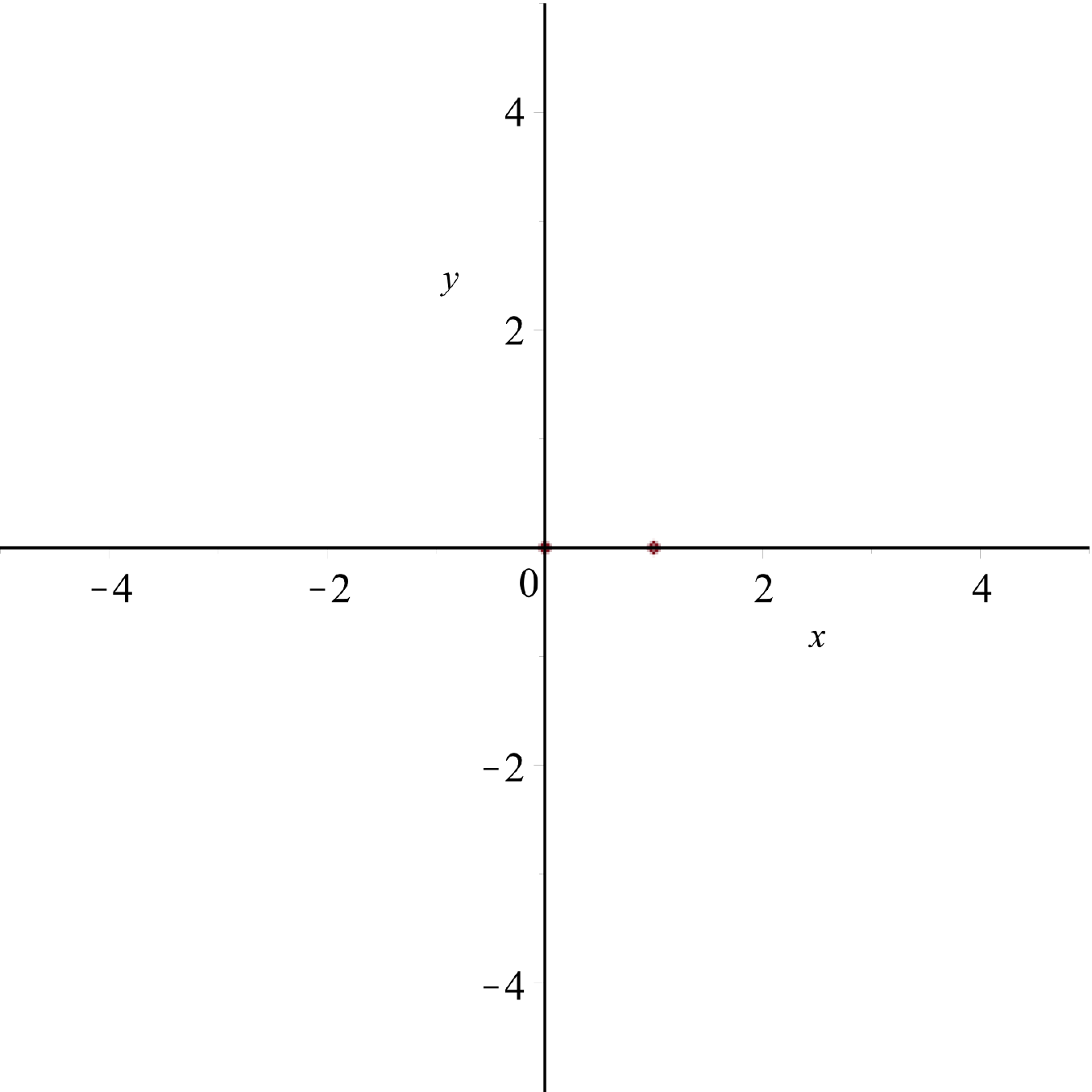}
  \includegraphics[height=.3\textheight]{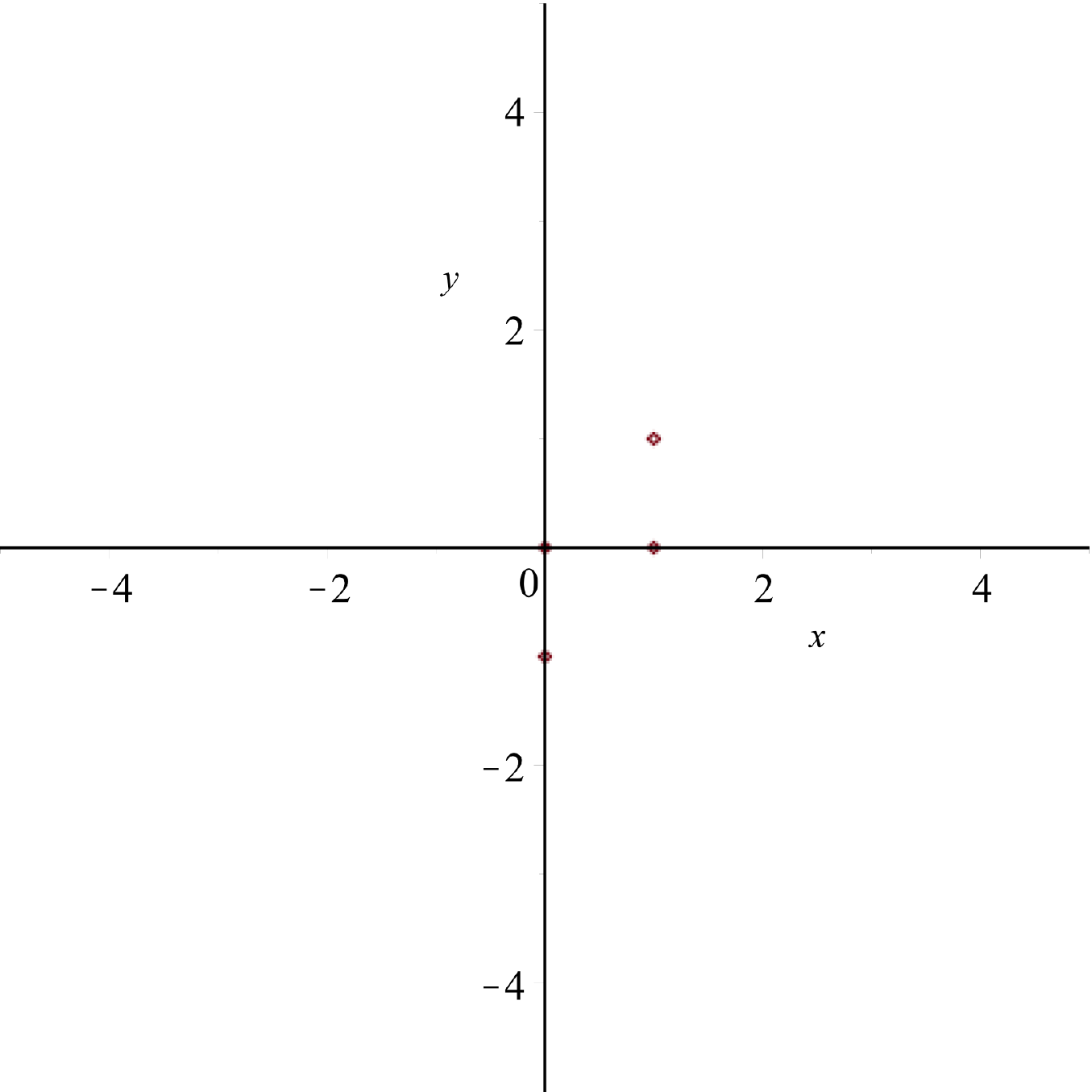}
  \includegraphics[height=.3\textheight]{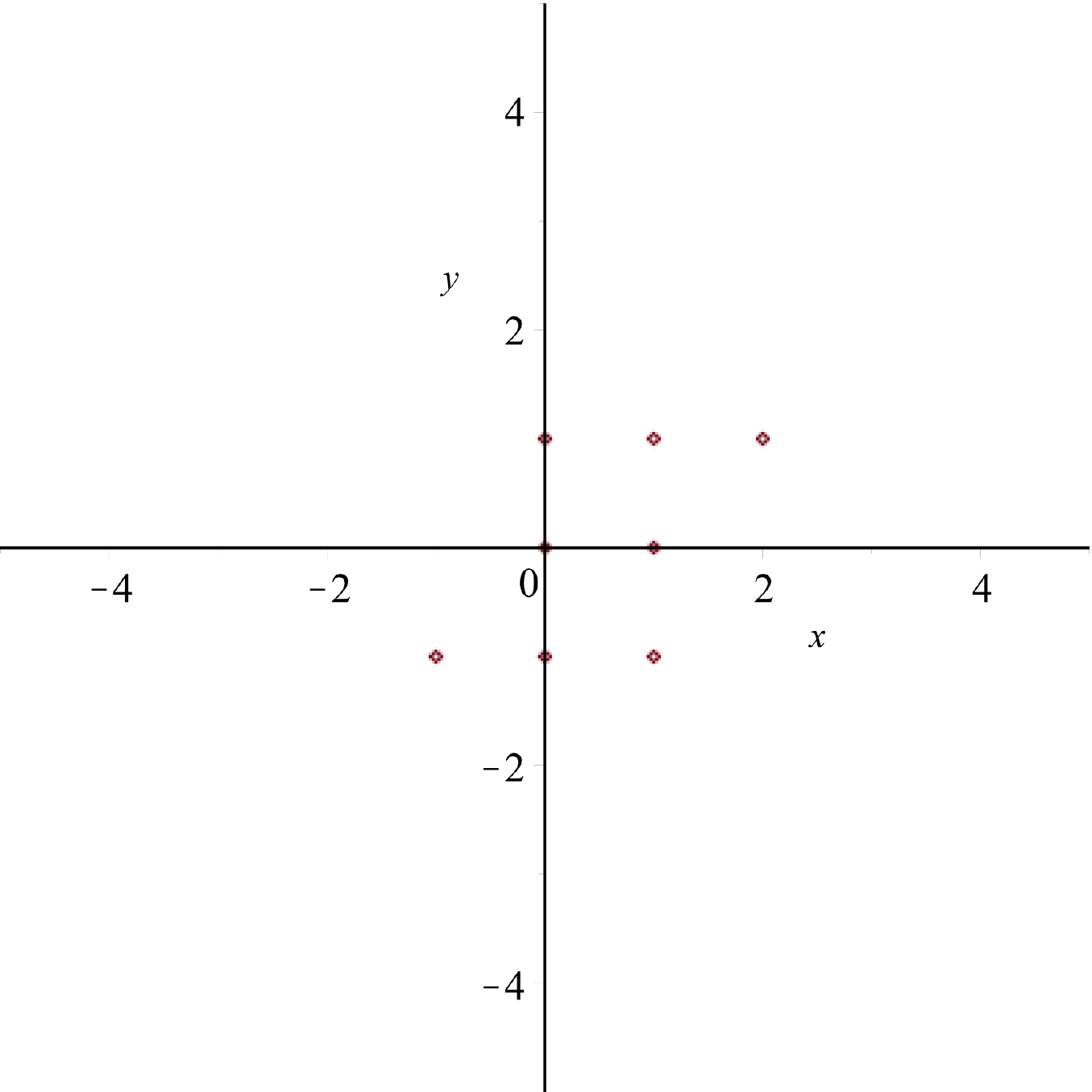}
  \includegraphics[height=.3\textheight]{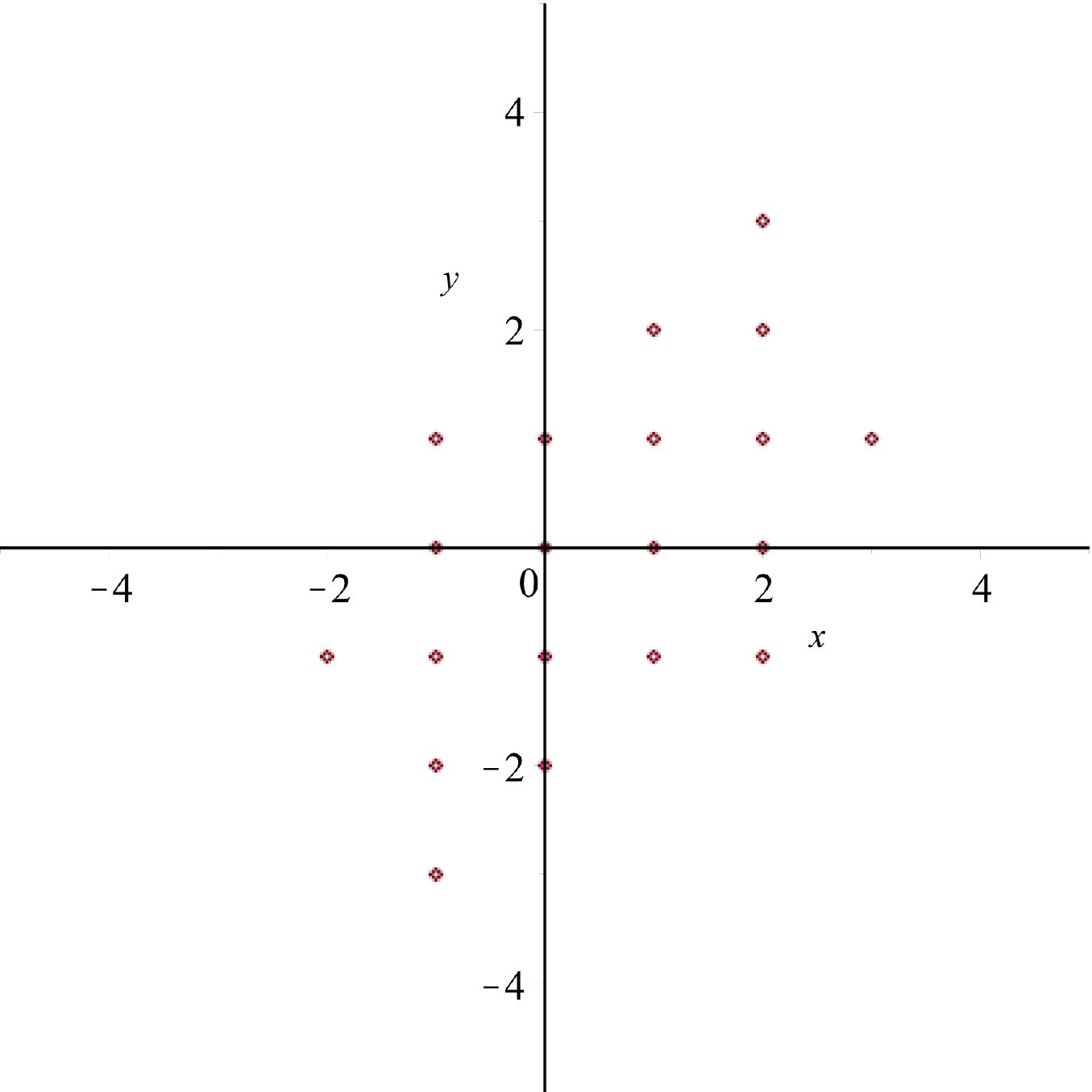}
  \includegraphics[height=.3\textheight]{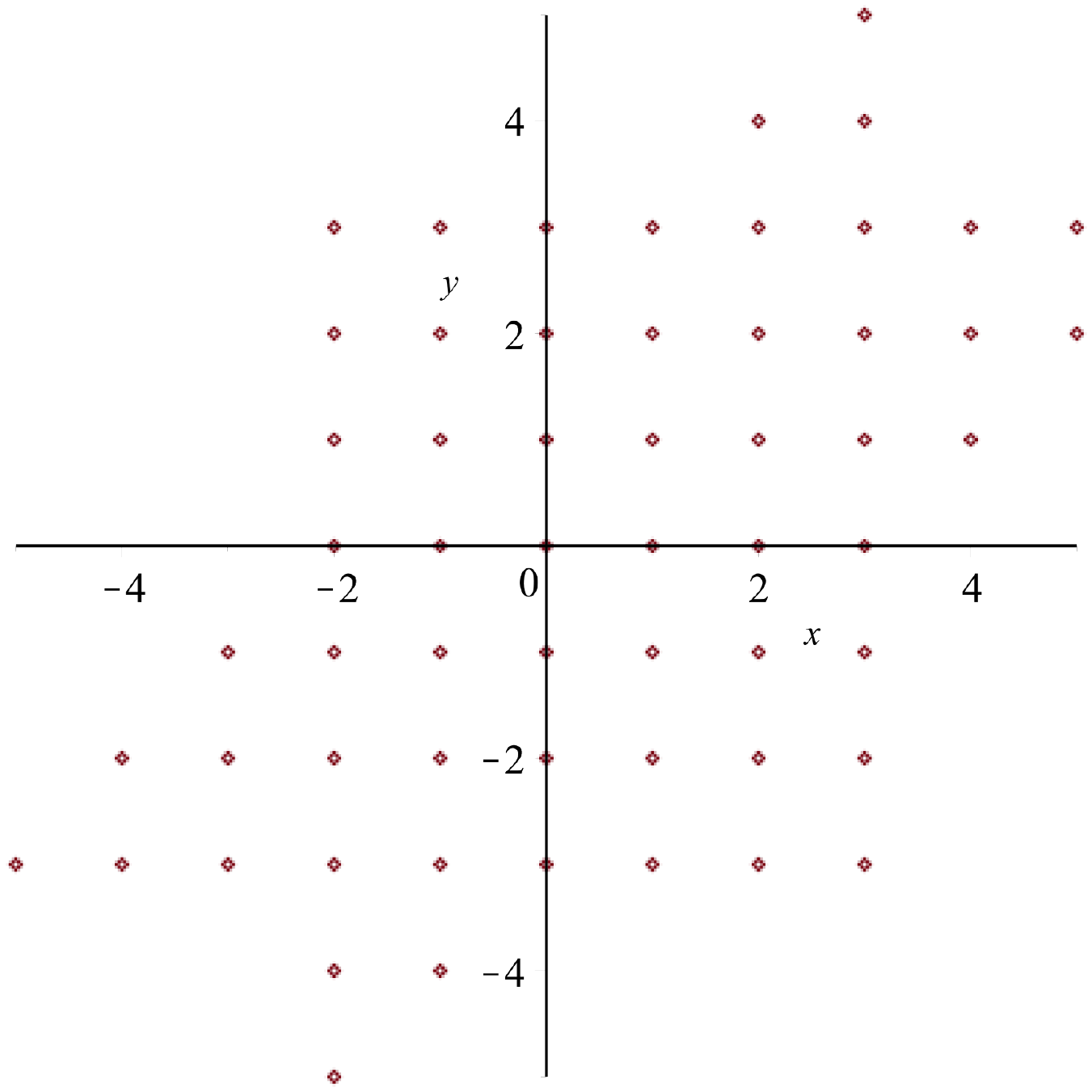}
  \includegraphics[height=.3\textheight]{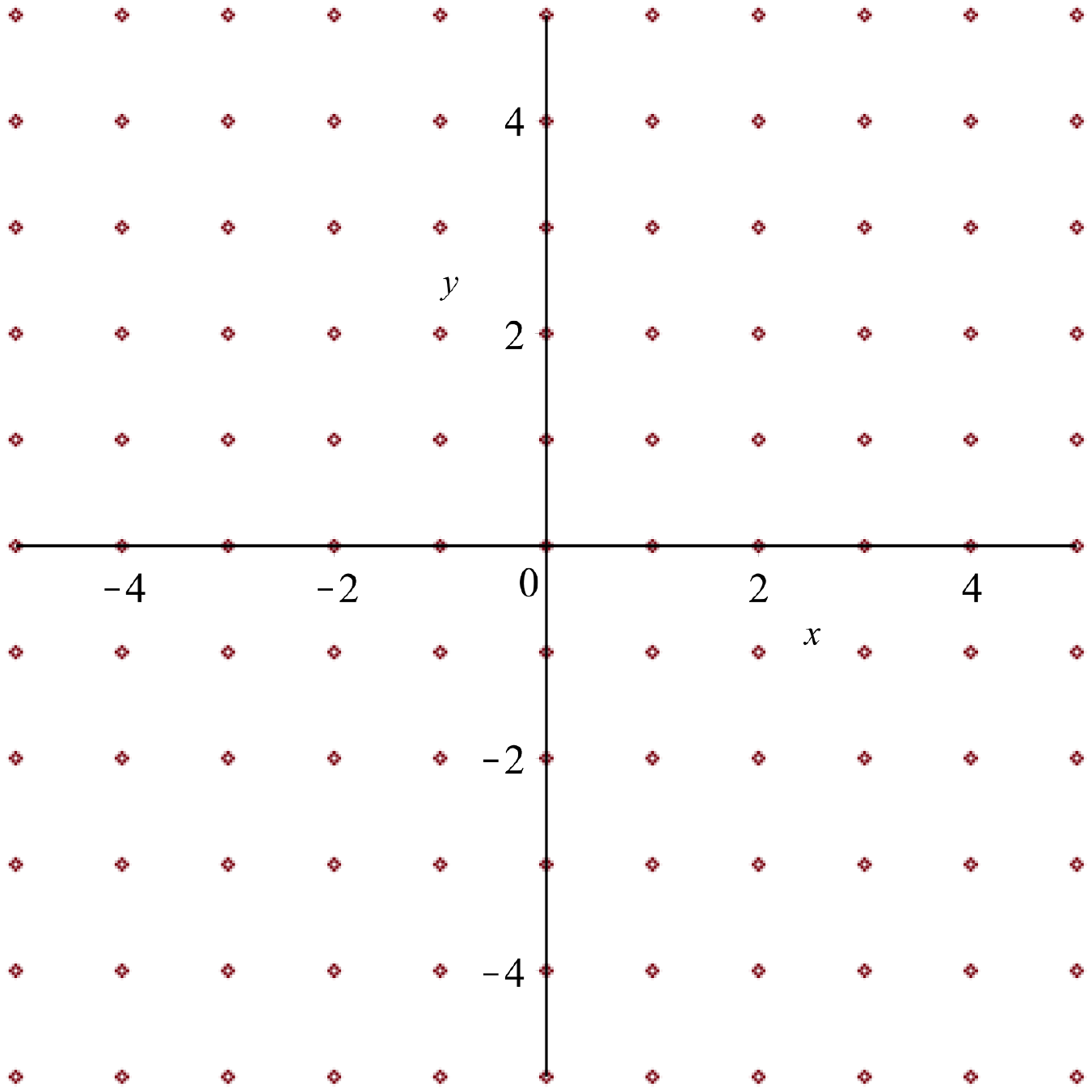}
  \caption[Generational Expansion]{Each successive graph shows all possible intersections from the previous graph using $U=\{1,e^{i\pi/4},i\}$ as our set of allowable angles.}\label{fig: gauss}
\end{figure}

\begin{rem} All the figures in this document were created by coding the algorithm for iterated intersections into Maple \cite{maple}.
\end{rem}

Notice that a pattern seems to emerge: the points constructed all have the form $a+bi$ where $a,b \in \Z$. This seems to indicate that this origami construction generates the Gaussian integers, a subring of $\mathbb{C}$. In fact, this is a special case of the main result of \cite{nedrenco}, where $z=i$. We prove in the next section that the ring of algebraic integers of an imaginary quadratic field can always be obtained through an origami construction.

\section{Constructing $\mathcal{O}(\mathbb{Q}(\sqrt{m}))$}

A natural question, related to the previous section is this: Which subrings of $\C$ can be generated through an origami construction, that is, which subrings are origami rings? We have seen that the cyclotomic integers $\Z[\zeta_n]$, where $n$ is prime, are origami rings by Theorem \ref{Buhler}. 

Let $m<0$ be a square-free integer, so $\Q(\sqrt{m})$ is an imaginary quadratic field. Denote by $\mathcal{O}(\Q(\sqrt{m}))$ the ring of algebraic integers in $\Q(\sqrt{m})$. Recall that a complex number is an algebraic integer if and only if it is the root of some monic polynomial with coefficients in $\Z$. Then we have the following well-known theorem (for details see, for example, \cite[pg. 15]{marcus}). 

\begin{theorem}
The set of algebraic integers in the quadratic field $\Q (\sqrt{m})$ is 
\begin{align*}
\{a+b\sqrt{m}: a,b\in\Z\} & \text{ if $m\equiv 2$ or $3 \mod 4$}\\
\left\{\frac{a+b\sqrt{m}}{2}: a,b\in\Z , a\equiv b \mod 2\right\} & \text{ if $m\equiv 1 \mod 4$}
\end{align*}
\end{theorem}

And so, we can state our main theorem. 

\begin{theorem}\label{ourtheorem}
Let $m<0$ be a squarefree integer, and let $\theta=\arg(1+\sqrt{m})$. Then $\mathcal{O}(\Q(\sqrt{m}))=R (U)$ where
\begin{enumerate}
\item $U=\{1,i,e^{i\theta}\}$, if $m\equiv 2$ or $3 \mod 4$.
\item $U=\{1, e^{i\theta},e^{i(\pi-\theta)}\}$, if $m\equiv 1 \mod 4$.  
\end{enumerate}
\end{theorem}

 Notice that the Gaussian integers are a special case of Theorem \ref{ourtheorem}.1.

\subsection{Proof of Theorem \ref{ourtheorem}.1}

Let $m\equiv 2$ or $3 \mod 4$ and $m<0$. Let $U=\{1,i,e^{i\theta}\}$ where $\theta$ is the principal argument of $1+\sqrt{m}$. 

\begin{lemma}\label{cases1}
$I_{u,v} (p,q)\in \Z[\sqrt{m}]$ whenever $u,v\in U$ and $p,q\in \Z[\sqrt{m}] $. 
\end{lemma}

\begin{proof}
Since there are three possible directions, there are ${3\choose 2}=6$ cases to consider. Let $p=a+b\sqrt{m}$ and $q=c+d\sqrt{m}$. Then
\begin{enumerate}
\item $I_{1,i} (p,q)=c+b\sqrt{m}$\\
\item $I_{1,e^{i\theta}} (p,q)=b+c-d+b\sqrt{m}$\\
\item $I_{i,1} (p,q)=a+d\sqrt{m}$\\
\item $I_{i,e^{i\theta}} (p,q)=a+ (a-c+d)\sqrt{m}$\\
\item $I_{e^{i\theta},1} (p,q)=a-b+d+d\sqrt{m}$\\
\item $I_{e^{i\theta},i} (p,q)=c+ (-a+b+c)\sqrt{m}$\\
\end{enumerate}

In other words, if $p,q\in\Z[\sqrt{m}]$, then so is $I_{u,v} (p,q)$. All of these can be obtained from straightforward computations using equation (\ref{closedform}). For example, 
$$I_{1,i} (p,q)=\frac{[1,a+b\sqrt{m}]}{[1,i]}i+\frac{[i,c+d\sqrt{m}]}{[i,1]}=c+b\sqrt{m}.$$

\end{proof}

\par This concludes the proof for the closure of the intersection operator. In other words, as long as our seed set starts with elements in $\Z[\sqrt{m}]$, then the intersections will also be in $\Z[\sqrt{m}]$. We can also express this claim as $R (U)\subseteq \Z[\sqrt{m}]$. It remains to be shown that any element in $\Z[\sqrt{m}]$ is also an element in $R (U)$.

\begin{lemma}
$\Z[\sqrt{m}]\subseteq R (U)$.
\end{lemma}

\begin{proof}

\par Let $a+b\sqrt{m}$ be an element in in $\Z[\sqrt{m}]$. We want to show that it can be constructed from starting with $\{0,1\}$ and the given set $U$. 

%Notice that if we can construct $b\sqrt{m}$ and $ (1+b)\sqrt{m}$ from iterated intersections by starting at $\{0,1\}$, and we can construct $a$ by starting at $\{0,1\}$, then we can construct $a+b\sqrt{m}$ by starting at $\{b\sqrt{m},  (1+b)\sqrt{m}\}$ (essentially we're ``shifting" the construction of $a$). 

We can reduce the problem by showing that given points $\{n+k\sqrt{m},n+1+k\sqrt{m}\}$ we can construct $$n-1+k\sqrt{m}, n+2+k\sqrt{m}, n+ (k-1)\sqrt{m} \hspace{3pt} \text{and}  \hspace{3pt} n+1+ (k+1)\sqrt{m}.$$ In Figure \ref{redux} we give an illustration of this step of the proof. In essence, the following is the induction step to a double induction on the real and imaginary components of an arbitrary integer we are constructing. That is, we prove that for any two adjacent points in the construction, we can construct points that are adjacent in every direction. This, and the fact that our seed is $\{0,1\}$, is enough to show that we can construct all of the integers. 

\begin{figure}[!h]
  \includegraphics[ height=0.3\textheight]{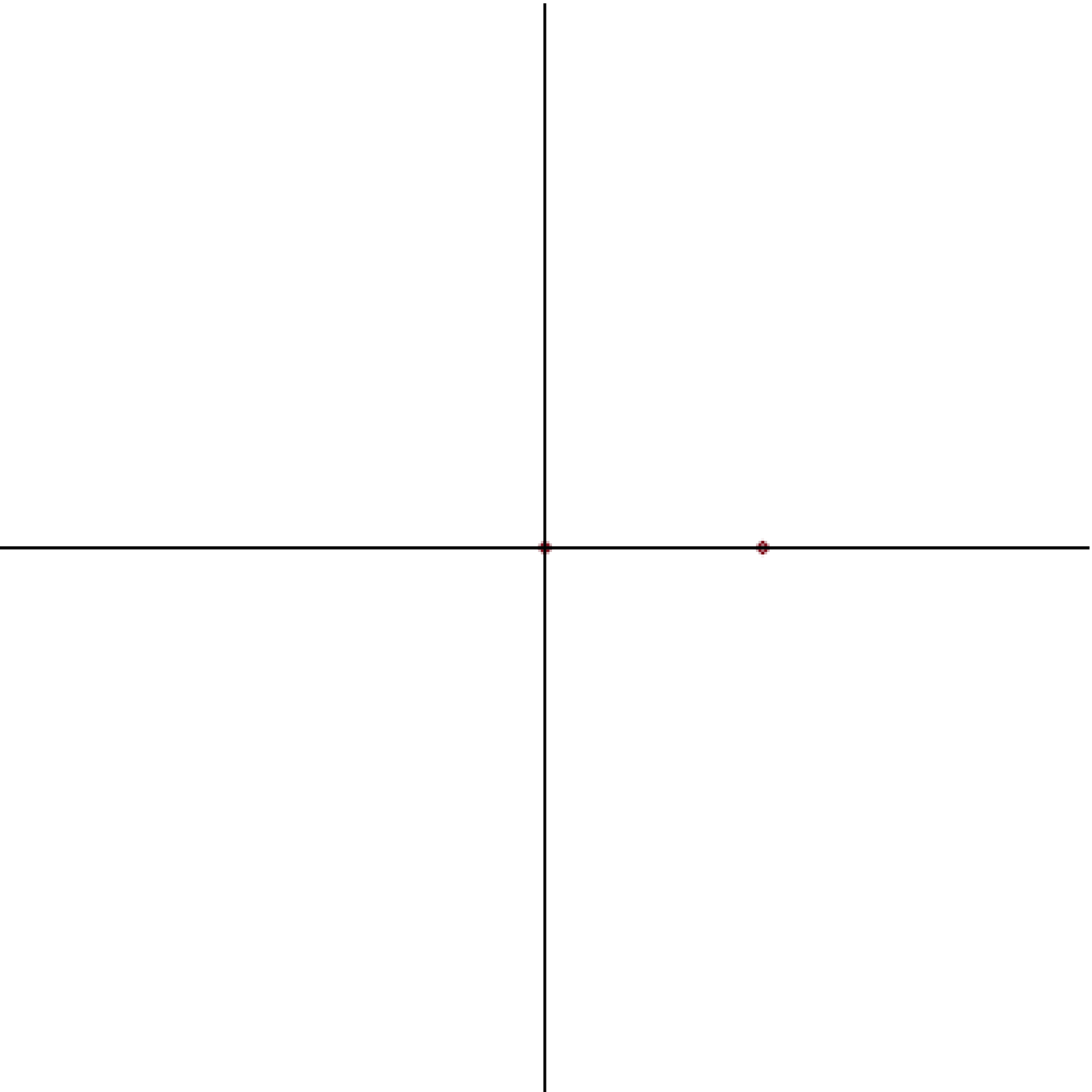}
  \includegraphics[ height=0.3\textheight]{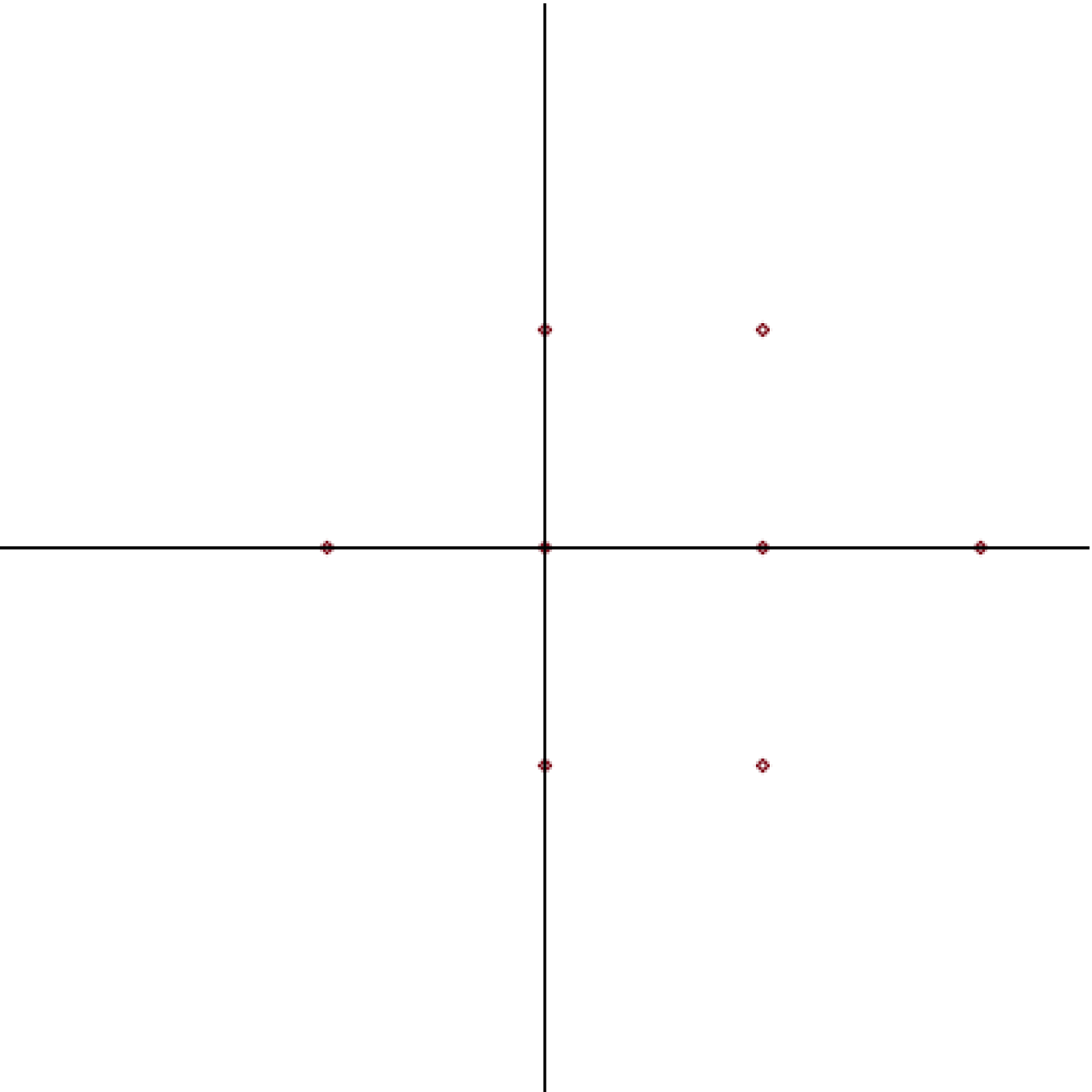}
  \caption[Reduction]{Given two points next to each other, we want to show that we can generate all the points immediately around them. This results in more points next to each other, upon which we can repeat the process. Notice that there are six adjacent points to an adjacent pair and we only prove this for four, but it is easy to see that this is enough. Think of this as two overlapping ``crosses". }\label{redux}
\end{figure}

We will now construct the desired points using the appropriate reference points.

\begin{description}
\item[Constructing $n+2+k\sqrt{m}$] Consider $$I_{i,1} (I_{1,e^{i\theta}} (I_{e^{i\theta},i} (n+k\sqrt{m},n+1+k\sqrt{m}),n+1+k\sqrt{m}),n+1+k\sqrt{m}).$$ Notice that we can evaluate this expression using the six cases enumerated in Lemma \ref{cases1}. In particular, we apply case  (6) first to get  $$I_{i,1} (I_{1,e^{i\theta}} (n+1+ (k+1)\sqrt{m},n+1+k\sqrt{m}),n+1+k\sqrt{m})$$ Next, we apply case  (2) to get  $$I_{i,1} (n+1+ (k+1)\sqrt{m},n+1+k\sqrt{m})$$ Finally, we use case  (3) to get $$n+2+k\sqrt{m}$$
\item[Constructing $n-1+k\sqrt{m}$] Consider $$I_{i,1} (I_{1,e^{i\theta}} (I_{i,e^{i\theta}} (n+k\sqrt{m},n+1+k\sqrt{m}),n+k\sqrt{m}),n+k\sqrt{m})$$ First, we apply case  (4) to get $$I_{i,1} (I_{1,e^{i\theta}} (n+ (k-1)\sqrt{m},n+k\sqrt{m}),n+k\sqrt{m})$$ Next, we apply case  (2) to get $$I_{i,1} (n-1+ (k-1)\sqrt{m},n)$$ Finally, we apply case  (3) to get $$n-1+k\sqrt{m}$$
\item[Constructing $n+ (k+1)\sqrt{m}$ and $n+1+ (k+1)\sqrt{m}$] Consider $$I_{e^{i\theta},i} (n+k\sqrt{m},n+1+k\sqrt{m})$$ Using case  (6) we get $$n+1+ (k+1)\sqrt{m}$$ Now consider $$I_{i,1} (n+k\sqrt{m},n+1+ (k+1)\sqrt{m})$$ Using case  (3) we get $$n+ (k+1)\sqrt{m}$$
\item[Constructing $n+ (k-1)\sqrt{m}$ and $n+1+ (k-1)\sqrt{m}$] Consider $$I_{i,e^{i\theta}} (n+k\sqrt{m},n+1+k\sqrt{m})$$ Using case  (4) we get $$n+ (k-1)\sqrt{m}$$ Consider $$I_{i,1} (n+1+k\sqrt{m},n+ (k-1)\sqrt{m})$$ Using case  (3) we get $$n+1+ (k-1)\sqrt{m}$$
\end{description}

\end{proof}

With this we have shown that $\Z[\sqrt{m}]\subseteq R (U)$, completing the proof.

\subsection{Proof of \ref{ourtheorem}.2}

The proof for \ref{ourtheorem}.2 employs the same strategy as the proof for \ref{ourtheorem}.1, with a subtle  difference given by the slightly different structure of the ring.

Let $m\equiv 1 \mod 4$ and $m<0$. Let $U=\{1, e^{i\theta},e^{i (\pi-\theta)}\}$ where $\theta$ is the principal argument of $1+\sqrt{m}$. 

\begin{lemma}
 $I_{u,v} (p,q)\in \mathcal{O}(\Q(\sqrt{m}))$ where $u,v\in U$ and $p,q\in \mathcal{O}(\Q(\sqrt{m}))$.
 \end{lemma}
 
 \begin{proof}
 
  Again, there are six cases to consider. Let $p=\frac{a+b\sqrt{m}}{2}$ and $q=\frac{c+d\sqrt{m}}{2}$, where $a\equiv b\bmod 2$ and $c\equiv d\bmod2$. 

\begin{enumerate}
\item $I_{1,e^{i\theta}} (p,q)=b+c-d+b\sqrt{m}$\\
\item $I_{1,e^{i (\pi-\theta)}} (p,q)=c+d-b+b\sqrt{m}$\\
\item $I_{e^{i\theta},1} (p,q)=a-b+d+d\sqrt{m}$\\
\item $I_{e^{i\theta},e^{i (\pi-\theta)}} (p,q)=\frac{ (a-b+c+d)+ (b-a+c+d)\sqrt{m}}{2}$\\
\item $I_{e^{i (\pi-\theta)},1} (p,q)=a+b-d+d\sqrt{m}$\\
\item $I_{e^{i (\pi-\theta)},e^{i\theta}} (p,q)=\frac{ (a+b+c-d)+ (a+b-c+d)\sqrt{m}}{2}$\\
\end{enumerate}

All of these cases can be obtained, again, by straightforward computation using (\ref{closedform}), and are left as a exercise. Notice that (1), (2), (3), and (5) all are clearly in the ring of algebraic integers. The only additional fact to show is that (4) and (6) are as well. But it's easy to see that $$a-b+c+d\equiv b-a+c+d \bmod 2,$$  since $a\equiv b\bmod 2$. And similarly $$a+b+c-d\equiv a+b-c+d \bmod 2,$$ because $c\equiv d\bmod 2$. 
\end{proof}

\par This concludes the proof for the closure of the intersection operator. In other words, as long as our seed set starts with elements in $\mathcal{O}(\Q(\sqrt{m}))$, then the intersections will also be in $\mathcal{O}(\Q(\sqrt{m}))$. That is, $R (U)\subseteq \mathcal{O}(\Q(\sqrt{m}))$. It remains to be shown that any element in $\mathcal{O}(\Q(\sqrt{m}))$ is also an element in $R (U)$.

\begin{lemma}
$\mathcal{O}(\Q(\sqrt{m}))\subseteq R(U)$.
\end{lemma}
\begin{proof}
\par Let $\dfrac{a+b\sqrt{m}}{2}$ be an element in in $\mathcal{O}(\Q(\sqrt{m}))$. As before, we can reduce the problem to one of double induction. This is done by showing that given points
 $$\left\{\frac{n+k\sqrt{m}}{2},\frac{n+2+k\sqrt{m}}{2}\right\}$$
  we can construct $$\dfrac{n-2+k\sqrt{m}}{2}, \dfrac{n+4+k\sqrt{m}}{2}, \dfrac{n+1+ (k-1)\sqrt{m}}{2} \hspace{3pt} \text{and}  \hspace{3pt}\dfrac{n+1+ (k+1)\sqrt{m}}{2}.$$ Figure \ref{redux2} is an illustration of the points we want to construct.
  
\begin{figure}[!h]
  \includegraphics[ height=0.2\textheight]{redux12}
  \includegraphics[ height=0.2\textheight]{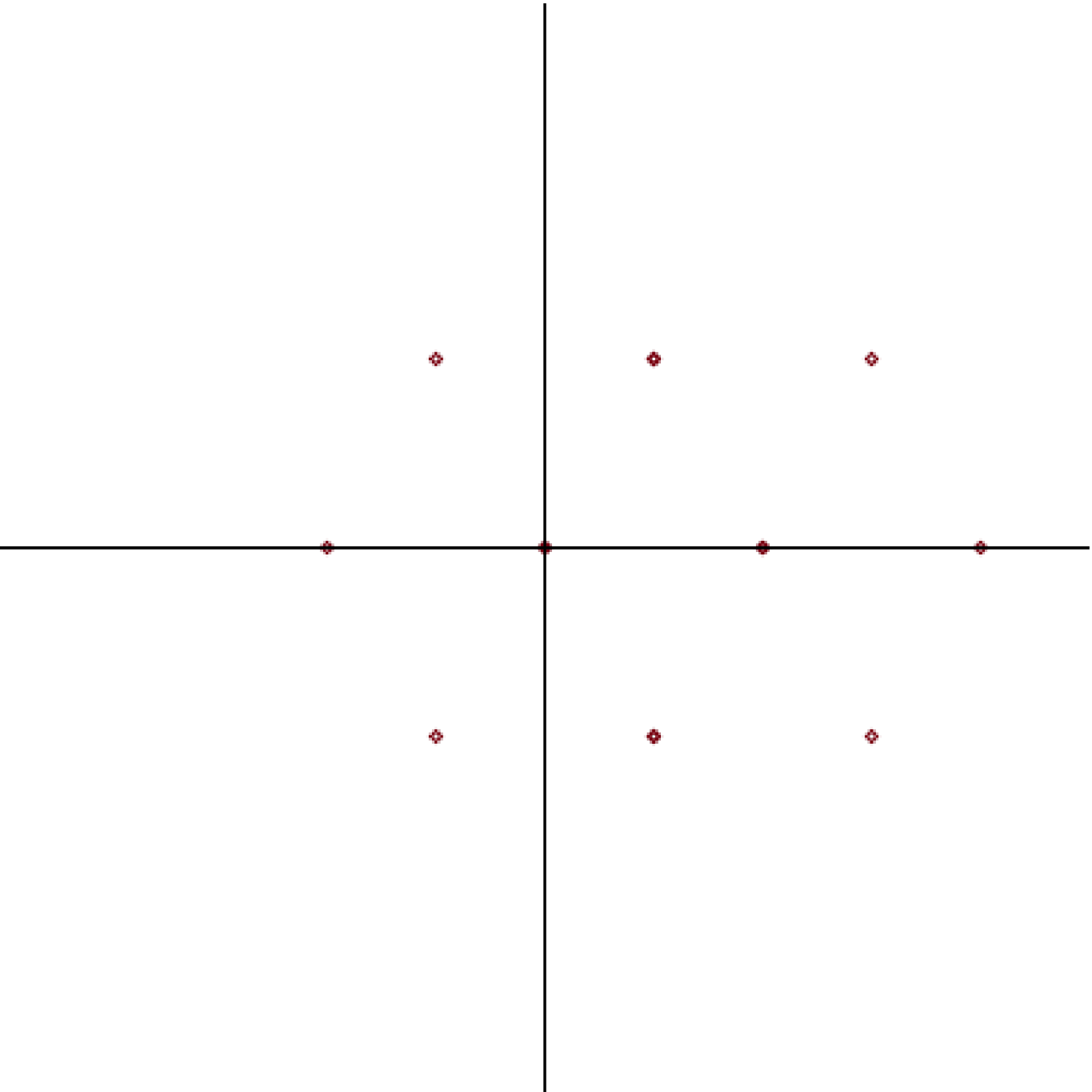}
  \includegraphics[ height=0.2\textheight]{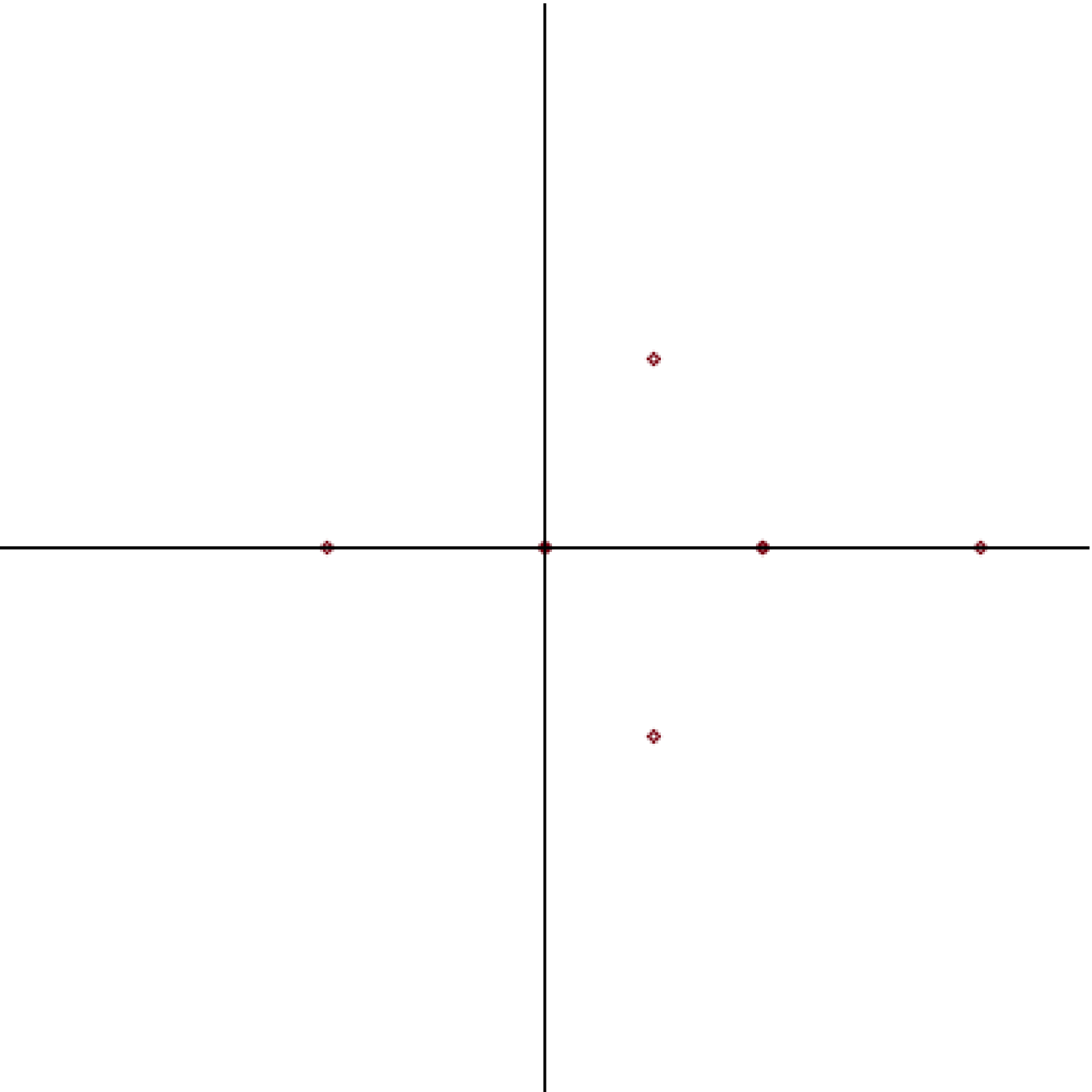}
  \caption[Reduction]{Given two points next to each other, we want to show that we can generate all the points immediately around them. Notice that this results in more points next to each other, upon which we can repeat the process. As before, we do not need to actually make all eight adjacent points, since it suffices to make the two points on either side of where we start, and the two points above and below where we start. The graph on the left illustrates the starting points. The graph in the center are all points adjacent to the two starting points. The graph on the right shows the points whose construction is sufficient to prove the theorem.}\label{redux2}
\end{figure}

  We will now show how to construct the desired points.

\begin{description}
\item[Constructing $\frac{n+1+ (k-1)\sqrt{m}}{2}$] Consider $$I_{e^{i \theta},e^{i (\pi-\theta)}} \left(\frac{n+k\sqrt{m}}{2},\frac{n+2+k\sqrt{m}}{2}\right)$$
By applying case  (4) from above, we see that $$I_{e^{i \theta},e^{i (\pi-\theta)}} \left(\frac{n+k\sqrt{m}}{2},\frac{n+2+k\sqrt{m}}{2}\right)=\frac{n+1+ (k-1)\sqrt{m}}{2}$$
\item[Constructing $\frac{n+1+ (k+1)\sqrt{m}}{2}$] Consider $$I_{e^{i (\pi-\theta)},e^{i \theta}} \left(\frac{n+k\sqrt{m}}{2},\frac{n+2+k\sqrt{m}}{2}\right)$$ 
By applying case  (6) from above, we see that $$I_{e^{i (\pi-\theta)},e^{i \theta}} \left(\frac{n+k\sqrt{m}}{2},\frac{n+2+k\sqrt{m}}{2}\right)=\frac{n+1+ (k+1)\sqrt{m}}{2}$$

\item[Constructing $\frac{n-2+k\sqrt{m}}{2}$] Consider $$I_{e^{i\theta},1} \left(I_{1,e^{i (\pi-\theta)}} \left(\frac{n+1+ (k+1)\sqrt{m}}{2},\frac{n+k\sqrt{m}}{2}\right),\frac{n+k\sqrt{m}}{2}\right)$$
By applying case  (1) from above, we can reduce the previous expression to $$I_{e^{i\theta},1} \left(\frac{n-1+ (k+1)\sqrt{m}}{2},\frac{n+k\sqrt{m}}{2}\right)$$
We further reduce the expression using case  (5) from above. The result is $$I_{e^{i\theta},1} \left(\frac{n-1+ (k+1)\sqrt{m}}{2},\frac{n+k\sqrt{m}}{2}\right)=\frac{n-2+k\sqrt{m}}{2}$$
\item[Constructing $\frac{n+4+k\sqrt{m}}{2}$] Consider $$I_{e^{i (\pi-\theta)},1} \left(I_{1,e^{\pi\theta}} \left(\frac{n+1+ (k+1)\sqrt{m}}{2},\frac{n+2+k\sqrt{m}}{2}\right),\frac{n+2+k\sqrt{m}}{2}\right)$$
By applying case  (5) from above, we can reduce the previous expression to $$I_{e^{i (\pi-\theta)},1} \left(\frac{n+3+ (k+1)\sqrt{m}}{2},\frac{n+2+k\sqrt{m}}{2}\right)$$
We further reduce the expression using case  (1) from above. The result is $$I_{e^{i (\pi-\theta)},1} \left(\frac{n+3+ (k+1)\sqrt{m}}{2},\frac{n+2+k\sqrt{m}}{2}\right)=\frac{n+4+k\sqrt{m}}{2}$$
\end{description}

\end{proof}

With this we have shown that $$\left\{\frac{a+b\sqrt{m}}{2}: a,b\in\Z , a\equiv b \mod 2\right\}\subseteq R (U)$$ completing the proof.

\subsection{Some final illustrations and remarks}

In Figure \ref{fig: gausspartial} we see that when we use $U=\{1,i,e^{i\theta}\}$ as our angle set, then the origami ring grows into the first and third quadrants, and bleeds into the others. As discussed, $R(U)=\Z[i]$. 

In Figure \ref{fig: partialhalfs} we see that when we use $U=\{1,e^{i\arg(1+\sqrt{-3})},e^{i(\pi-\arg(1+\sqrt{-3}))}\}$ as our angle set, then the origami ring grows along the real line, and slowly bleeds into the rest of the complex plane. 
This is an illustration of the second case of Theorem \ref{ourtheorem}. In this case, $R(U)=\mathcal{O}(\Q(\sqrt{-3}))$. 

Of course, $R(U)$ is assumed to be closed under the intersection operator, so the growth pattern doesn't matter in an abstract sense. However, computationally, it means that the number of steps it takes to construct a point is not related to that point's modulus. In fact, we get an entirely different measure of distance if we only consider the number of steps it takes to generate a point. One possible additional exploration is, given more general starting angles and points, to describe the dynamics of the iterative process. 

These examples also serve to illustrate the progression of the iterative process, which was coded into Maple to produce the graphs. 

\begin{figure}[!h]
  \includegraphics[ height=.3\textheight]{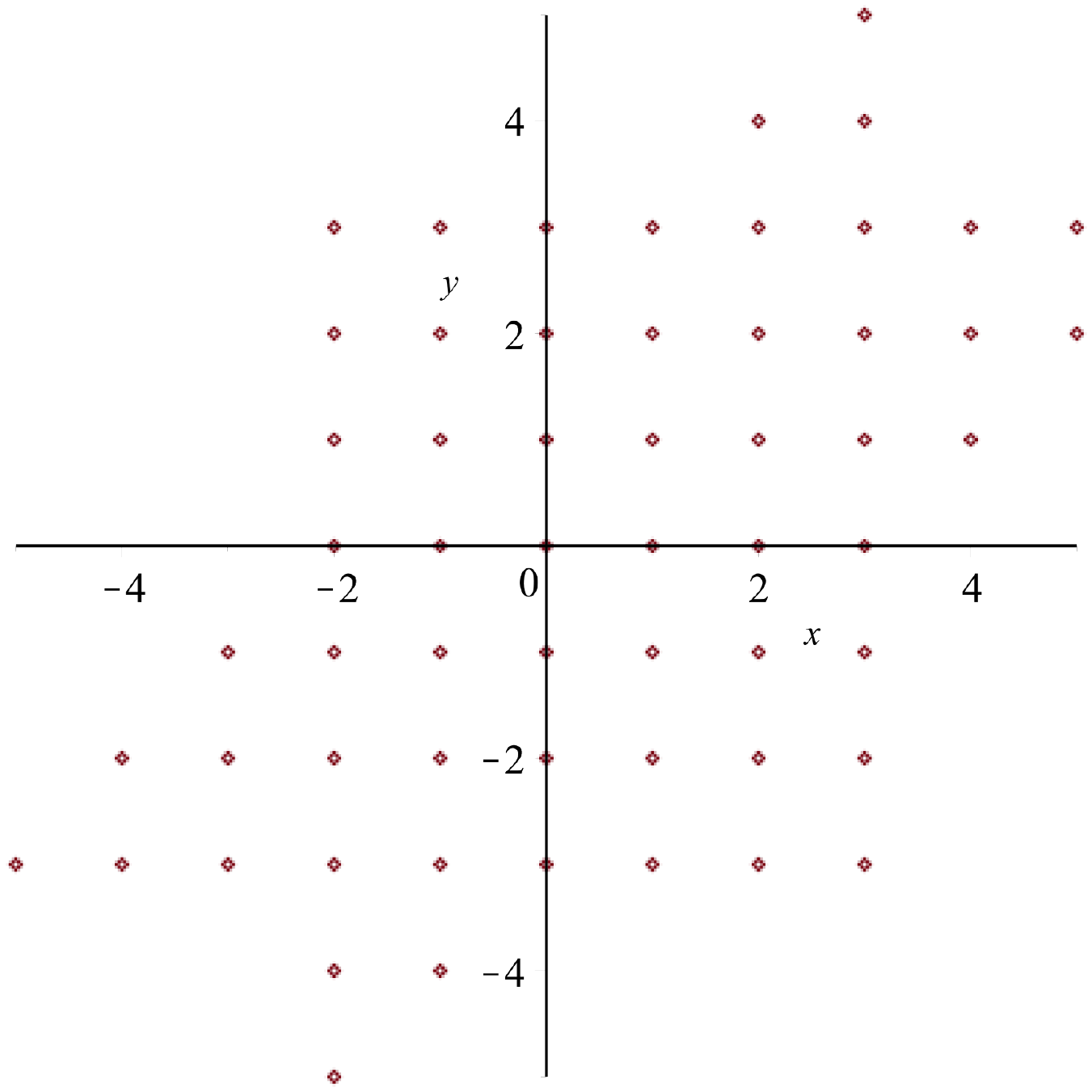}
  \caption[Five Generations of the Easy Case]{This graph depicts the first 5 generations of origami points using $U=\{1, i,e^{i\frac{\pi}{4}}\}$}\label{fig: gausspartial}
\end{figure}

\begin{figure}[!h]
  \includegraphics[ height=.3\textheight]{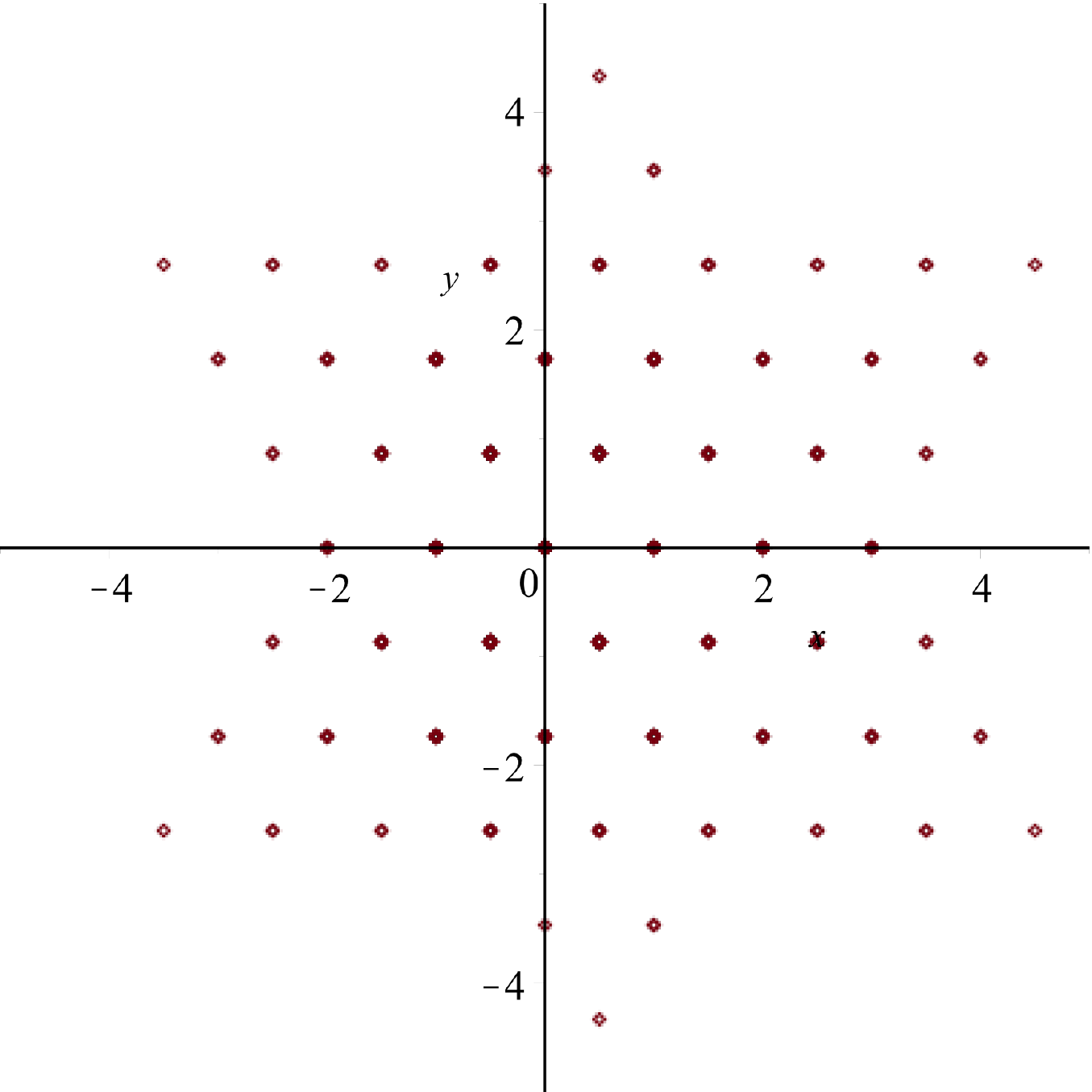}
  \caption[Five Generations of the Hard Case]{This graph depicts the first 5 generations of origami points using $$U=\{1,e^{i\arg(1+\sqrt{-3})},e^{i(\pi-\arg(1+\sqrt{-3}))}\}$$}\label{fig: partialhalfs}
\end{figure}
\pagebreak

%\bibliographystyle{empty} 
%\bibliography{thesis} 

\begin{thebibliography}{10}


\bibitem{Buhler2010}
	Buhler, J.,  Butler, S., de Launey, W., and Graham, R.
	\emph{Origami rings},
	Journal of the Australian Mathematical Society \textbf{92} (2012), 299-311.	

\bibitem{maple} Maple 17. Maplesoft, a division of Waterloo Maple Inc., Waterloo, Ontario.

\bibitem{marcus} Marcus, Daniel A. \emph{Number Fields}. New York, New York: Springer Verlag, 1997. 

\bibitem{nedrenco} Nedrenco, Dmitri, \emph{On origami rings}, arXiv: 1502.07995v1, Feb 2015. 






\end{thebibliography}

\end{document}